\documentclass[12pt]{amsart}
\usepackage{amsmath,amssymb,amsfonts}
\usepackage[colorlinks=true,linkcolor=red,citecolor=blue]{hyperref}
\newtheorem{theorem}{Theorem}[section]
\newtheorem{proposition}{Proposition}[section]
\newtheorem{corollary}{Corollary}[section]

\newtheorem{definition}{Definition}[section]

\newcommand\myenum[1]

\begin{document}
\title[$\mathcal{S}$-L-weakly compact operators]{Contribution to operators between Banach lattices}

\author{Hassan Khabaoui}
\address{Hassan Khabaoui, Moulay Ismail University, Faculty of sciences, Department of Mathematics, B.P. 11201 Zitoune, Mekn\`{e}s, Morocco.}
\email{khabaoui.hassan2@gmail.com}

\author{Jawad H'michane}
\address{Engineering sciences Lab. ENSA, B.P 241, Ibn Tofail University, Kenitra, Morocco.}
\email{hm1982jad@gmail.com}

\author{Kamal El fahri}
\address{Kamal El fahri, Ibn Zohr University,  Faculty of sciences, Department of Mathematics, Laboratory of Mathematics and Application, Functional analysis team, Morocco.}
\email{kamalelfahri@gmail.com }
\begin{abstract}
In this paper we introduce and study a new class of operators related to norm bounded sets on Banach Lattice and which brings together several classical classes of operators (as o-weakly compact operators, b-weakly compact operators, M-weakly compact operators, L-weakly compact operators, almost Dunford-Pettis operators). As consequences, we give some new lattice approximation properties of these classes of operators.
	
\end{abstract}

\keywords{L-weakly compact set, L-weakly compact operator, order continuous Banach lattice, o-weakly compact operator, b-weakly compact operator, M-weakly compact operator,
almost Dunford Pettis operator}
\subjclass[2010]{46B42, 47B60, 47B65.}
\maketitle

\section{Introduction}
Along this paper $E$, $F$ mention Banach lattices, $X$, $Y$ are Banach spaces. The positive cone of $E$ will be denoted by $E^+$.\\
Recall that a net $(x_{\alpha})\subset E$ is unbounded absolutely weakly convergent (abb, uaw-convergent) to $x$ if $(|x_{\alpha} - x| \wedge u)$ converges weakly to zero for every $u\in E^{+}$, we write $x_{\alpha}\overset{uaw}\longrightarrow x$. We note that every disjoint sequence of a Banach lattice is uaw-null \cite[ Lemma 2] {Zabeti}. A net $(x'_{\alpha})$ is unbounded absolutely weak$^{*}$ convergent (abb, uaw$^{*}$-convergent) to $x'$ if $(|x'_{\alpha} - x'| \wedge u')$ converges weak$^{*}$ to zero for every $0\leq u'\in E'$, we write $x'_{\alpha}\overset{uaw^{*}}\longrightarrow x$.
Recall from \cite{Mey}  that a norm bounded subset $A$ of a Banach lattice $E$ is L-weakly compact if ${\underset{n\longrightarrow +\infty}{\lim}} \|x_n\| = 0$ for every disjoint sequence $(x_n)$ contained in $sol(A)$, where $sol(A):=\{x\in E: \exists y\in A \text{ with } |x|\leq\color{black} |y|\}$ is the solid hull of the set $A$. Alternatively, $A$ is L-weakly compact if and only if $|| x_n||\rightarrow 0$ for every norm bounded uaw-nul  sequence $(x_n)$ of $sol(A)$ \cite[Proposition 3.3] {k}.\\
In this paper, we introduce and study a new class of operators attached on a norm bounded subset of the starting space (Definition \ref{definition0}) and which groups together several classes of operators, as M-weakly compact operators (Corollary \ref{c9}), order weakly compact operators (Theorem \ref{prr}), b-weakly compact operators (Theorem \ref{prop0101}), almost Dunford-Pettis operators (Proposition \ref{p4}) and L-weakly compact operators (Corollary \ref{cooo}). As consequences, we obtain new characterizations of L-weakly compact sets (Corollary \ref{coo}), of order continuous Banach lattice (Corollary \ref{co100}), of KB-space (Corollary \ref{c200}) and of positive Schur property (Corollary \ref{schur}).

\section{Preliminaries and notations}


To state our results, we need to fix some notations and recall some definitions. A Banach lattice is a Banach space $(E,\parallel \cdot \parallel)$ such that $E$ is a vector lattice and its norm satisfies the following property: for each $x, y \in E$ such that $|x| \leq |y|$, we have $\parallel x\parallel	 \leq	\parallel y \parallel$. $E$ is order continuous if for each net $(x_{\alpha})$ such that $x_{\alpha} \downarrow 0$ in $E$, the net $(x_{\alpha})$  converges to 0 for the norm 	$ \parallel \cdot \parallel$, where the notation $x_{\alpha} \downarrow 0$ means that the net $(x_{\alpha})$ is decreasing, its infimum exists and $inf(x_{\alpha}) = 0$. \\
We will use the term operator $T:E \longrightarrow F$ from  $E$ to $F$ to mean a bounded linear mapping. $T'$ will be the adjoint operator of $T:E \longrightarrow F$ defined from $F'$  into $E'$ by $T'(f)(x)= f(T (x))$ for each $f \in F'$ and each $x \in E$. An operator $T:E \longrightarrow F$ is positive if $T (x) \in F^+$ whenever $x \in E^+$. For more information on positive operators see the
book of Aliprantis-Burkinshaw \cite{AB}\\
We need to recall definitions of the following operators:
\begin{enumerate}
\item  An operator $T : X \longrightarrow F$ is said to be L-weakly compact, if $T (B_X)$ is an
L-weakly compact subset of $F$.
\item  An operator $T : E \longrightarrow X$ is said to be order weakly compact, if $T ([0, x])$ is
a relatively weakly compact subset of $X$ for every $x$ in $E^{+}$.
\item  An operator $T : E \longrightarrow X$ is said to be b-weakly compact, if $T ([0, x]\cap E)$ is
a relatively weakly compact subset of $X$ for every $x \in (E'')^+$.
\item  An operator $T : E \longrightarrow X$ is said to be  M-weakly compact, if $T(x_n)\overset{||.||}\longrightarrow 0$, for every norm bounded disjoint sequence $(x_n)$ in $E$.
\item An operator $T : E \longrightarrow X$ is said to be almost Dunford-Pettis, if  $T(x_n)\overset{||.||}\longrightarrow 0$ for every disjoint weakly null sequence $(x_n)$ in $E$.
\end{enumerate}

\section{Main results}

We start this section by the following definition.

\begin{definition}\label{definition0}
 Let $\mathcal{S}$ be a norm bounded subset of $E$. An operator $T:E\longrightarrow Y$ is said to be $\mathcal{S}$-L-weakly compact (abb, $\mathcal{S}$-Lwc) if for every uaw-null sequence $(x_n) \subset sol(\mathcal{S})$, we have $T(x_n)\overset{||.||}\longrightarrow 0$.
\end{definition}
Observing that for a norm bounded subset $\mathcal{S}$ of $E$,  $Id_E$ is  $S$-Lwc if and only if $S$ is an L-weakly compact subset of $E$ (\cite[Proposition 3.3] {k}) and that $T$ is M-weakly compact if and only if $T$ is a $B_E$-Lwc \cite[Corollary 3.1] {k}, where $B_E$ denotes the closed unit ball of $E$.

For a norm bounded subset $\mathcal{S}$ of $E$, we note by  $LWC_{\mathcal{S}}(E,Y)$ the space of all $\mathcal{S}$-Lwc operators from $E$ into $Y$. It is a norm closed vector subspace of $L(E,Y)$, the space of all operators from $E$ into $Y$, and it is a left ideal in $L(E,Y)$.  In particular,  if $\mathcal{S}$ is an L-weakly compact subset of $E$, then every operator $T$ defined from $E$ to $Y$ is $\mathcal{S}$-Lwc. On the other hand, note that if $A$, $B$ are two norm bounded subsets of $E$ such that $A \subset B$ and $T$ is an operator from $E$ into $Y$, then $T$ is $A$-Lwc whenever $T$ is $B$-Lwc. On the other hand, $T$ is $B$-Lwc if and only if $T$ is $sol(B)$-Lwc.

\begin{proposition}\label{pro8}
	Let $T:E\longrightarrow Y$ be an operator and  $\mathcal{S}$ be a norm bounded subset of $E$. If $T$ is $\mathcal{S}$-Lwc, then  for every $ \lambda\in \mathbb{R} $ the operator $T$ is $\lambda \mathcal{S}$-Lwc.
\end{proposition}
\begin{proof}
It follows from the fact that for every $ \lambda\in \mathbb{R} $,  $sol(\lambda \mathcal{S})=\lambda sol(\mathcal{S})$.
\end{proof}

\begin{proposition}\label{pro4}
Let $T:E\longrightarrow Y$ be an operator and $A$, $B$  are norm bounded subsets of $E$. If $T$ is $A$-Lwc and $B$-Lwc, then $T$ is $(A+B)$-Lwc.
\end{proposition}
\begin{proof}
 Let $(x_n)$ be a  uaw-null sequence of $sol(A+B)$, then there exist two sequences $(a_n)\subset A$ and $(b_n) \subset B$ such that $x_{n}^{+} \leq |a_n+b_n|$. Therefore,
  by the Riesz decomposition property \cite[Theorem 1.13]{AB} there exist  tow positive  elements $a^1_n$ and $b^1_n$ satisfying $x_{n}^{+} =a^1_n+ b^1_n$, $|a^1_n| \leq |a_n|$ and $|b^1_n| \leq |b_n|$ for each $n$. So, $(a^1_n)$ is a uaw-null sequence of $sol(A)$ and $(b^1_n)$ is a uaw-null sequence of $sol(B)$. As $T$ is $A$-Lwc and $B$-Lwc, then $(Ta^1_n)\overset{||.||}\longrightarrow 0$ and $T(b^1_n)\overset{||.||}\longrightarrow 0$ and hence $T(x_{n}^{+})\overset{||.||}\longrightarrow 0$. By the same reason, we found $T(x_{n}^{-})\overset{||.||}\longrightarrow 0$. Therefore, $T(x_{n})\overset{||.||}\longrightarrow 0$. That is, $T$ is $(A+B)$-Lwc,  as
claimed.
\end{proof}

As immediate consequences of the previous result, we have the following results.
\begin{corollary}\label{c10}
Let $T:E\longrightarrow Y$ be an operator and $A$, $B$ are norm bounded subsets of $E$. We have the following statements:
\begin{enumerate}
\item If $T$ is $A$-Lwc and $B$-Lwc, then for every $ (\lambda,\mu)\in \mathbb{R}^2 $, we have that $T$ is  $(\lambda A+ \mu B)$-Lwc.
\item If $T$ is $A$-Lwc and $B$-Lwc, then $T$ is $(A\cup B)$-Lwc.
\item If $T$ is $A$-Lwc or $B$-Lwc, then $T$ is  $(A\cap B)$-Lwc.
\end{enumerate}
\end{corollary}

\begin{proposition}\label{pro6}
Let $T$ be an operator  from  $E$ to $Y$, $\mathcal{S}$ be a norm bounded subset of $E$ and $I$ be the ideal generated by $\mathcal{S}$. If $T$ is $\mathcal{S}$-Lwc, then for each $x \in I$ we have that $T$ is $[-|x|,|x|]$-Lwc.
\begin{proof}
Let $\mathcal{S}$ be a norm bounded subset of $E$, $I$ be the ideal generated by $\mathcal{S}$ and $ x \in I $, then there exist $\alpha > 0$ and  some vectors $x_1, .... , x_n \in \mathcal{S}$ with $ |x| \leq \alpha \sum_{i=1}^{n}|x_i|$. By the Riesz decomposition property \cite[Theorem 1.13]{AB} we have $$[-|x|,|x|] \subset \alpha [-|x_1|,|x_1|]+ .... + \alpha [-|x_n|,|x_n|].$$
We observe that, for each $i=1,.....,n$ we have $[-|x_i|,|x_i|] \subset sol(\mathcal{S})$, then  $$[-|x|,|x|] \subset \underbrace{\alpha sol(\mathcal{S})+ .... + \alpha sol(\mathcal{S})}_{n\text{-times}}.$$ Since $T$ is $\mathcal{S}$-Lwc, then $T$ is $sol(\mathcal{S})$-Lwc and so by Corollary \ref{c10}, we infer that $T$ is $\underbrace{\alpha sol(\mathcal{S})+ .... + \alpha sol(\mathcal{S})}_{n\text{-times}}$-Lw.  Therefore,  $T$ is $[-|x|,|x|]$-Lwc.
\end{proof}
\end{proposition}

\begin{theorem}\label{theo2}
	Let $T:E\longrightarrow Y$ be an operator and $u \in E^{+}$. Then, the following statements are equivalent:
	\begin{enumerate}
	\item $T$ is $[-u,u]$-Lwc.
		\item For each $\varepsilon>0$, there exists some $g \in (E')^+$ such that $$\forall x \in [-u,u], \,\,\,||T(x)||\leq g(|x|)+\varepsilon.$$
	\end{enumerate}
\end{theorem}
\begin{proof}
	$1)\Rightarrow 2)$ Let $\varepsilon>0$ and $u \in E^{+}$. As $T$ is $[-u,u]$-Lwc, we see that every disjoint sequence of $[0, u]$ converges uniformly to zero on $T'(B_{Y'})$, then by \cite[Theorem 4.40] {AB} there exists some $g \in (E')^+$ such that $$(|T'(f)| - g)^{+}(u) <\varepsilon \quad \text{ holds for all } f\in B_{Y'}.$$ Let $x \in [-u,u]$, then  for each $f\in B_{Y'}$ we have
	$$|f(T(x))| \leq |T'(f)|(|x|)\leq g(|x|) + \varepsilon,$$ so $$||T(x)||\leq g(|x|) + \varepsilon$$

	$2)\Rightarrow 1)$ Let $\varepsilon>0$ and $(x_n)$ be a uaw-null sequence of $[-u,u]$. We have to show that
$Tx_n\overset{||.||}\longrightarrow 0$.  By our hypothesis, there exists some $g\in (E')^{+}$  such that $$||T(x_n)||\leq g(|x_n|)+\frac{\varepsilon}{2} \text{ for every } n\in \mathbb{N}.$$ As $x_n\overset{uaw}\longrightarrow 0$ in $E$, then $g(|x_n|)=g(|x_n|\wedge u)\rightarrow 0$ and hence there exists some integer $m$ such that $g(|x_n|)\leq \frac{\varepsilon}{2}$ for every $n\geq m $. So, for every $n\geq m$ we have $||T(x_n)||\leq \varepsilon$ which implies that $T(x_n)\overset{||.||}\longrightarrow 0$. Therefore, $T$ is $[-u,u]$-Lwc.
\end{proof}

In the following result, we present some characterizations of $\mathcal{S}$-Lwc operators.

\begin{theorem}\label{theor}
	For an operator $T:E\longrightarrow Y$ and a norm bounded subset of $\mathcal{S}\subset E$, the following statements are equivalent:
	\begin{enumerate}
		\item $T$ is $\mathcal{S}$-Lwc.
		\item For each $\varepsilon>0$, there exist some $g\in (E')^+$ and $ u\in E^{+} $  such that $$||T(x)||\leq g(|x|\wedge u)+\varepsilon \quad \text{ for all } x \in sol(\mathcal{S}).$$
	\item For every uaw-null net $(x_{\alpha})$ of $sol(\mathcal{S})$, we have $T(x_{\alpha})\overset{||.||}\longrightarrow 0$.
	\end{enumerate}
\end{theorem}
\begin{proof}
	$1)\Rightarrow 2)$
	Let $\varepsilon > 0$; since $T$ is $\mathcal{S}$-Lwc, then by \cite[Theorem 4.36] {AB} there exists some $u \in E^{+}$ lying in the ideal generated by $sol(\mathcal{S})$ such that
	$$||T((|x|-u)^{+})||\leq \frac{\varepsilon}{4} \quad \text{ for all } x\in sol(\mathcal{S}).$$
	This implies that,
	$$||T(|x|)||\leq||T(|x|\wedge u)||+\frac{\varepsilon}{4} \quad \text{ for all }  x\in sol(\mathcal{S}).$$
	Hence,
	for every $x\in sol(\mathcal{S})$, we have $$||T(x^{+})||\leq ||T(x^{+}\wedge u)||+\frac{\varepsilon}{4}$$ and $$||T(x^{-})||\leq ||T(x^{-}\wedge u)||+\frac{\varepsilon}{4}.$$
	This implies that for every $x\in sol(\mathcal{S})$, we have $$||T(x)||\leq ||T(x^{+}\wedge u)||+||T(x^{-}\wedge u)||+\frac{\varepsilon}{2}\,\,\,\, (*).$$
	On the other hand, by  Proposition \ref{pro6} $T$ is $[-u,u]$-Lwc, and so by  Theorem \ref{theo2}, there exists some $g\in (E')^+$ such that for every $x\in sol(\mathcal{S})$, we have $$||T(x^{+}\wedge u)||\leq g(x^{+}\wedge u)+\frac{\varepsilon}{4}$$ and $$||T(x^{-}\wedge u)||\leq g(x^{-}\wedge u)+\frac{\varepsilon}{4}.$$
	On the other hand, using the fact that $ |x|\wedge u= x^{+}\wedge u + x^{-}\wedge u $, we see that for every $x\in sol(\mathcal{S})$ we have $$||T(x^{+}\wedge u)||+ ||T(x^{-}\wedge u)||\leq g(|x|\wedge u)+\frac{\varepsilon}{2}\,\,\,\, (**).$$
	Therefore, by combining $(*)$ and $(**)$, we have
$$||T(x)||\leq  g(|x|\wedge u)+\varepsilon \text{ for every } x\in sol(\mathcal{S}).$$
	$2)\Rightarrow 3)$ Let $\varepsilon>0$ and $(x_\alpha)$ a net of $sol(\mathcal{S})$. By our hypothesis, there exist some $g\in (E')^+$ and $u \in E^{+}$ such that,
	$$||T(x_\alpha)||\leq g(|x_\alpha|\wedge u)+\frac{\varepsilon}{2} \text{ for all } \alpha.$$
	Since $x_\alpha\overset{uaw}\longrightarrow 0 $, then there exists some $\alpha_0$ such that $g(|x_\alpha|\wedge u)< \frac{\varepsilon}{2}$ for every $\alpha \geq \alpha_0$.  Hence, for every $\alpha \geq \alpha_0$ we have $||T(x_\alpha)||\leq \varepsilon $, which implies that $T(x_\alpha)\overset{||.||}\longrightarrow 0$. \\
	$3)\Rightarrow 1)$ Obvious.
\end{proof}

As a consequence of the previous theorem, we have the following  characterizations of M-weakly compact operators which is exactly the \cite[ Corollary 3.1] {k}.
\begin{corollary}\label{c9}
For an operator $T:E\longrightarrow F$ the following statements are equivalent:
	\begin{enumerate}
		\item $T$ is M-weakly compact.
		\item For each $\varepsilon>0$, there exist some $g\in (E')^+$ and $ u\in E^{+} $  such that $$||T(x)||\leq g(|x|\wedge u)+\varepsilon \text{ for all } x\in B_E.$$
		\item For every  net $(x_{\alpha}) \subset B_E$ such that $x_{\alpha}\overset{uaw}\longrightarrow 0$, we have
		$T(x_{\alpha})\overset{||.||}\longrightarrow 0$.
		\item For every  sequence $(x_n) \subset B_E$ such that $x_n\overset{uaw}\longrightarrow 0$, we have $T(x_n)\overset{||.||}\longrightarrow 0$.
	\end{enumerate}
\end{corollary}

In a similar way, we may prove the following result which  present a dual version of the Theorem \ref{theor}. The proof of this theorem is similar to that of the Theorem \ref{theor}, so we will avoid copying it here.

\begin{theorem}\label{thoerem002}
Let $T:E\longrightarrow F$ be an operator and $\mathcal{S}$ a norm bounded subset of $F'$. Then, the following statements are equivalent:
	\begin{enumerate}
		\item $T'$ is $\mathcal{S}$-Lwc.
		\item For each  $\varepsilon>0$ there exist some $g\in (F')^+$ and $ u\in F^{+} $ such that $$||T'(f)||\leq (|f|\wedge g)(u)+\varepsilon \text{ for all } f \in sol(\mathcal{S}).$$
			\item For every uaw$^*$-null net $(f_{\alpha}) \subset  sol(\mathcal{S})$, we have $T'(f_{\alpha})\overset{||.||}\longrightarrow 0$.
	\item For every uaw$^*$-null sequence $(f_n) \subset sol(\mathcal{S})$, we have $T'(f_n)\overset{||.||}\longrightarrow 0$.
	\end{enumerate}
\end{theorem}

As consequence of the Theorem \ref{thoerem002}, we obtain new characterizations of L-weakly compact operators.
\begin{corollary}\label{cooo}
Let $T:E\longrightarrow F$ be an operator. Then, the following statements are equivalent:
	\begin{enumerate}
		\item $T$ is L-weakly compact.
		\item For each  $\varepsilon>0$ there exist some $g\in (F')^+$ and $ u\in F^ {+} $  such that $$||T'(f)||\leq (|f|\wedge g)(u)+\varepsilon \text{ for all } f\in B_{F'}.$$
		\item For every  net $(f_{\alpha}) \subset B_{F'}$ such that $f_{\alpha}\overset{uaw^*}\longrightarrow 0$, we have $T'(f_{\alpha})\overset{||.||}\longrightarrow 0$.
		\item For every  sequence $(f_n) \subset B_{F'}$ such that $f_n\overset{uaw^*}\longrightarrow 0$, we have $T'(f_n)\overset{||.||}\longrightarrow 0$.
	\end{enumerate}
\end{corollary}

Recall that a bounded subset $A$ in Banach lattice $E$ is said uaw-compact whenever every net
$(x_{\alpha})$  in $E $ has a subnet, which is uaw-convergent. Note that the standard basis $(e_n)$ of $\ell_1$ is uaw-null in $\ell_1$, then the set $A=\left\lbrace e_{n}: n\in \mathbb{N } \right\rbrace$ is relatively uaw-compact; but  $A$ is not relatively compact.

	\begin{proposition}\label{p}		
		Let $\mathcal{S}$ be a norm bounded subset of $E$. For an $\mathcal{S}$-Lwc operator $T:E\longrightarrow Y$, we have
		\begin{enumerate}
			\item $T(A)$ is compact for every norm bounded uaw-compact subset $A\subset sol(\mathcal{S})$.	
		\item $T(A)$ is relatively compact for every norm bounded relatively uaw-compact subset $A\subset sol(\mathcal{S})$.
		 \end{enumerate}
	\end{proposition}
\begin{proof}
	$1)$ Let $A$ be a  uaw-compact subset of $sol(\mathcal{S})$ and let $(x_{\alpha})$ be a net of $A$, then there exist a subnet  of $(x_{\alpha})$  which we denoted by $(x_{\alpha})$ and $x \in A$ such that $x_{\alpha} \overset{uaw}\longrightarrow x$. As $T$ is $\mathcal{S}$-Lwc, then by  Corollary \ref{c10} and Theorem \ref{theor} , $T(x_{\alpha} )\overset{||.||}\longrightarrow T(x)$. Therefore, $T(A)$ is a compact subset of $Y$.

	$2)$ Let $A$ be a relatively uaw-compact subset of $sol(\mathcal{S})$ and let $(x_{\alpha})$ be a net of $A$, then $(x_{\alpha})$ has a uaw-convergent subnet which we denoted by $(x_{\alpha})$. Since the subnet $(x_{\alpha})$ is norm bounded and uaw-Cauchy, then  the double net $(x_{\alpha}-x_{\beta})_{(\alpha,\beta)}$ is uaw-null. As $T$ is $\mathcal{S}$-Lwc, then by  Theorem \ref{theor} and Corollary \ref{c10} the net $(T(x_{\alpha}))$ is norm Cauchy and hence $(T(x_{\alpha}))$ is norm convergent. Therefore, $T(A)$ is a relatively compact subset of $Y$.
\end{proof}
In the following result we give characterizations of  $\mathcal{S}$-Lwc operators which are uaw-continuous.

\begin{proposition}\label{prop0001}
	Let $\mathcal{S}$ be a norm bounded subset of $E$. For a uaw-continuous operator $T:E\longrightarrow F$ the following statements are equivalent:
\begin{enumerate}
	\item $T$ is $\mathcal{S}$-Lwc.
	\item $T(A)$ is relatively compact for every relatively  uaw-compact subset $A\subset sol(\mathcal{S})$.
	\item $T(A)$ is compact for every   uaw-compact subset $A\subset sol(\mathcal{S})$.

\end{enumerate}
\end{proposition}
\begin{proof}

	$1)\Rightarrow 2)$ Follows from the Proposition \ref{p}.\\
	$2)\Rightarrow 3)$ Obvious.\\
	$3)\Rightarrow 1)$ Let $(x_n)$ be a uaw-null sequence of $sol(\mathcal{S})$, then  the set $A=\left\lbrace x_{n}: n\in \mathbb{N } \right\rbrace \cup\left\lbrace 0\right\rbrace $ is  uaw-compact, and hence it follows from our hypothesis that $T(A)$ is relatively compact. So, there exist a subsequence $(x_{\phi(n)})$ of $(x_n)$ and $y \in T(A)$ such that $T(x_{\phi(n)}) \overset{||.||}\longrightarrow y$. Since the sequence $(x_n)$ is uaw-null, then  $x_{\phi(n)} \overset{uaw}\longrightarrow 0$ and since $T$ is uaw-continuous then $y=0$. Therefore, $T(x_{n}) \overset{||.||}\longrightarrow 0$ which implies that $T$ is $\mathcal{S}$-Lwc.

\end{proof}

An immediate consequence of the previous result is the following.
\begin{corollary}
	For a uaw-continuous operator $T:E\longrightarrow F$, the following statements are equivalent:
\begin{enumerate}	
\item $T$ is M-weakly compact.
	\item $T(A)$ is  compact for every   uaw-compact subset $A\subset B_E$.
	\item $T(A)$ is relatively compact for every relatively  uaw-compact subset $A\subset B_E$.
\end{enumerate}
\end{corollary}


As another consequence of  Theorem \ref{theor} and Proposition \ref{prop0001}, we have the following characterizations of L-weakly compact sets.
\begin{corollary}\label{coo}
For a norm bounded subset $\mathcal{S}$ of $E$, the following statements are equivalent:
	\begin{enumerate}
		\item $\mathcal{S}$ is L-weakly compact.
		\item For each $\varepsilon>0$, there exist some $g\in (E')^+$ and $ u\in E^{+} $ such that $$||x||\leq g(|x|\wedge u)+\varepsilon \text{ for every } x \in sol(\mathcal{S}).$$
		\item For every  net $(x_{\alpha}) \subset  sol(\mathcal{S})$ such that $x_{\alpha}\overset{uaw}\longrightarrow 0$, we have $x_{\alpha}\overset{||.||}\longrightarrow 0$.
		\item Every norm bounded   uaw-compact subset $A\subset sol(\mathcal{S})$ is compact.
			\item Every norm bounded  relatively   uaw-compact subset $A\subset sol(\mathcal{S})$ is relatively compact.
	\end{enumerate}
\end{corollary}

Note that from \cite[Proposition 3.4]{k}, it is easy to see that an operator $T:E\longrightarrow Y$ is order weakly compact if and only if
for every $v\in E^{+}$, the operator $T$ is $\{v\}$-Lwc.

With the help of  Theorem \ref{theor} and Proposition \ref{p}, we are now in a position to
present new characterizations of the  order weakly compact operators.

\begin{theorem}\label{prr}
	For an operator $T:E\longrightarrow Y$, the following statements are equivalent:
	\begin{enumerate}
		\item $T$ is order weakly compact.	
		\item For every relatively compact subset $\mathcal{S}$ of $E$, the operator $T$ is $\mathcal{S}$-Lwc.
		\item For every $v\in E^{+}$, the operator $T$ is $\{v\}$-Lwc.
		\item For every $\varepsilon>0$ and $v\in E^+$, there exist some $g\in (E')^+$ and $ u\in E^{+} $ such that $$||T(x)||\leq g(|x|\wedge u)+\varepsilon \text{ for every } x \in [-v,v].$$
		\item For every $ v\in E^+$ and for every  net $(x_{\alpha}) \subset [-v,v]$ such that $x_{\alpha}\overset{uaw}\longrightarrow 0$, we have $T(x_{\alpha})\overset{||.||}\longrightarrow 0$.
		\item For every $ v\in E^+$ and for every  sequence $(x_n) \subset [-v,v]$ such that $x_n\overset{uaw}\longrightarrow 0$, we have $T(x_n)\overset{||.||}\longrightarrow 0$.
		\item  For every order bounded  uaw-compact subset $A$ of $E$, $T(A)$ is  compact.
		\item  For every order bounded relatively  uaw-compact subset $A$ of $E$, $T(A)$ is relatively compact.	
	\end{enumerate}
\end{theorem}

\begin{proof}
	It remains to show $1)\Rightarrow 2)$ and 	$8)\Rightarrow 1)$, the other implications are already seen before in the preceding results.

	$1)\Rightarrow 2)$ Let $T$ be an order weakly compact operator, $\mathcal{S}$ a relatively compact set of $E$ and $(x_n) $ a uaw-nul sequence of $sol(\mathcal{S})$ and let $\varepsilon > 0$. Since $(x_n) \subset sol(\mathcal{S})$ and $\mathcal{S}$ a relatively compact set of $E$, then there exist $u \in E^{+}$ and a sequence $(y_n) \subset \mathcal{S} $ such that
	$$||(x_{n}^{+}-u)^{+}||\leq ||(|y_n|-u)^{+}||\leq\frac{\varepsilon}{4||T||} \quad \text{ for all } n,$$
	which implies that for all $n$, we have $$||T(x_{n}^{+})||\leq ||T(x_{n}^{+}\wedge u)||+\frac{\varepsilon}{4}.$$
	Since $T$ is   order weakly compact then, $T$ is $[-u,u]$-Lwc operator, thus by Theorem \ref{theo2}, there exists $g\in (E')^+$ such that $$||T(x_{n}^{+}\wedge u)||\leq g(x_{n}^{+}\wedge u)+\frac{\varepsilon}{4} \text{ for all } n,$$
	So for every $n$ we have $$||T(x_{n}^{+})||\leq g(x_{n}^{+}\wedge u)+\frac{\varepsilon}{2}\quad \text{ for all } n.$$
	Since $x_n\overset{uaw}\longrightarrow 0 $, then there exists $n_0$ such that $g(x_{n}^{+}\wedge u)< \frac{\varepsilon}{2}$ for all $n\geq n_0 $. Thus, for each $n \geq n_0$ we have $||T(x_{n}^{+})||\leq \varepsilon $, which implies that $T(x_{n}^{+})\overset{||.||} \longrightarrow 0$.
	 By the same reason, we found $T(x_{n}^{-})\overset{||.||}\longrightarrow 0$. Hence, $T(x_{n})\overset{||.||}\longrightarrow 0$, as desired.

	 	 	$8)\Rightarrow 1)$ Let $v\in E^+ $ and $(x_n) \subset [-v,v]$ be a uaw-null sequence. The set $A=\left\lbrace x_{n}: n\in \mathbb{N } \right\rbrace  $ is  relatively uaw compact, implies that $(T(x_n))$ is relatively compact. By observing that the sequence $(x_n)$ is weakly null we infer that $T(x_{n}) \overset{||.||}\longrightarrow 0$.\\
\end{proof}

As a consequence of Theorem \ref{prr}, we obtain the following  characterizations of order continuous Banach lattices.

\begin{corollary}\label{co100}
 The following statements are equivalent:
\begin{enumerate}
	\item  $E$ is order continuous.	
	\item Every relatively compact subset $\mathcal{S}$ of $E$ is L-weakly compact.
	\item For every $x\in E^{+},\,\, Id_E \,\,\text{is} \, $\{x\}$\text{-Lw}$.
	\item For every $\varepsilon>0$ and $v\in E^+$, there exist some $g\in (E')^+$ and $ u\in E^{+} $ such that$$||x||\leq g(|x|\wedge u)+\varepsilon \text{ for all } x \in [-v,v].$$
	\item For every $ v\in E^+$ and for every  net $(x_{\alpha}) \subset [-v,v]$ such that $x_{\alpha}\overset{uaw}\longrightarrow 0$, we have $x_{\alpha}\overset{||.||}\longrightarrow 0$.
	\item For every $ v\in E^+$ and for every  sequence $(x_n) \subset [-v,v]$ such that $x_n\overset{uaw}\longrightarrow 0$, we have $x_n\overset{||.||}\longrightarrow 0$.	
	\item Every order bounded  uaw-compact subset $A\subset E$ is compact.
		\item Every order bounded  relatively  uaw-compact subset $A\subset E$ is relatively compact.
\end{enumerate}
\end{corollary}

Recall from \cite{b-semicompact} that a subset $A$ of a Banach lattice $E$ is said to be b-semi compact
 if it is almost order bounded as a subset of $E^{\prime
\prime }$, that is, for every $\varepsilon >0$ there exists $u\in
E_{+}^{\prime \prime }$ such that $A\subset \left[ -u,u\right] +\varepsilon
B_{E^{\prime \prime }}$.

By repeating the proof of Theorem \ref{prr}, we can prove a similar result for b-weakly compact operators as follows.

\begin{theorem}\label{prop0101}
	For an operator $T:E \longrightarrow Y$, the following statements are equivalent:
	\begin{enumerate}
\item  $T$ is b-weakly compact.	
\item For every b-semi compact subset $\mathcal{S}$ of $E$, the operator $T$ is $\mathcal{S}$-Lwc.
\item For every $v \in (E'')^+,\,\, T \,\,\text{is} \, ([-v,v]\cap E)\text-{Lwc}$.
\item For each  $\varepsilon>0$ and $ v\in (E'')^+ $, there exist some $g\in (E')^+$ and $ u\in E^{+} $  such that $$||T(x)||\leq g(|x|\wedge u)+\varepsilon \text{ for all } x \in [-v,v]\cap E.$$
\item For every $v\in (E'')^+ $  and for every  net $(x_{\alpha}) \subset [-v,v]\cap E$ such that $x_{\alpha}\overset{uaw}\longrightarrow 0$, we have $T(x_{\alpha})\overset{||.||}\longrightarrow 0$.
\item For every $ v\in (E'')^+$ and for every  sequence $(x_n) \subset [-v,v]\cap E$ such that $x_n\overset{uaw}\longrightarrow 0$, we have $T(x_n)\overset{||.||}\longrightarrow 0$.
\item  For every b-order bounded   uaw-compact subset $A\subset E$, $T(A)$ is compact.
\item  For every b-order bounded  relatively uaw-compact subset $A\subset E$, $T(A)$ is relatively compact.
	\end{enumerate}
\end{theorem}

As a consequence of  Theorem \ref{prop0101}, we have the following characterizations of KB-spaces.
	\begin{corollary}\label{c200}
The following statements are equivalent:
	\begin{enumerate}
		\item  $E$ is a KB space.
		\item Every b-semi compact subset $\mathcal{S}$ of $E$  is L-weakly compact.			
		\item For every $v \in (E'')^+,\,\, Id_E \,\,\text{is } \, ([-v,v]\cap E)\text{-Lwc}$.
		\item For each  $\varepsilon>0$ and $v\in (E'')^+ $, there exist some $g\in (E')^+$ and $ u\in E^{+} $  such that $$||x||\leq g(|x|\wedge u)+\varepsilon \text{ for all }x \in [-v,v]\cap E.$$
\item For every $ v\in (E'')^+$ and for every  net $(x_{\alpha}) \subset [-v,v]\cap E$ such that $x_{\alpha}\overset{uaw}\longrightarrow 0$, we have $x_{\alpha}\overset{||.||}\longrightarrow 0$.
\item For every $ v\in (E'')^+$ and for every  sequence $(x_n) \subset [-v,v]\cap E$ such that $x_n\overset{uaw}\longrightarrow 0$, we have $x_n\overset{||.||}\longrightarrow 0$.	
\item Every b-order bounded    uaw-compact subset $A\subset E$ is  compact.
\item Every b-order bounded  relatively  uaw-compact subset $A\subset E$ is relatively compact.
		\end{enumerate}
\end{corollary}


In terms of relatively weakly compact sets and $\mathcal{S}$-Lwc operators the  almost Dunford-Pettis
operators are characterized as follows.

\begin{proposition}\label{p4}
	For an operator $T:E\longrightarrow Y$, the following statements are equivalent:
	\begin{enumerate}
		\item  $T$ is almost Dunford-Pettis.	
		\item  For every relatively weakly compact subset $\mathcal{S}$ of $E$, the operator $T$ is $\mathcal{S}$-Lwc.
		\item For every relatively weakly compact subset $\mathcal{S}$ of $E$ and for every  net $(x_{\alpha}) \subset sol(\mathcal{S})$ such that $x_{\alpha}\overset{uaw}\longrightarrow 0$, we have $T(x_{\alpha})\overset{||.||}\longrightarrow 0$.
		\item For every relatively weakly compact subset $\mathcal{S}$ of $E$ and for every  sequence $(x_n) \subset sol(\mathcal{S})$ such that $x_n\overset{uaw}\longrightarrow 0$, we have $T(x_n)\overset{||.||}\longrightarrow 0$.	
	\end{enumerate}
\end{proposition}
\begin{proof}

	$1)\Rightarrow 2)$ Let $\mathcal{S}$ be a relatively weakly compact subset of $E$ and $(x_n)$ be a uaw-null sequence of $sol(\mathcal{S})$,
then by \cite[Theorem 3.1] {k}, we have $x_n^{+}\overset{w}\longrightarrow 0$ and $x_n^{-}\overset{w}\longrightarrow 0$. On the other hand, since $T$ is almost Dunford-Pettis,
it follows from  \cite[Theorem 2.2] {aq} that $T(x_n^{+})\overset{||.||}\longrightarrow 0$ and $T(x_n^{-})\overset{||.||}\longrightarrow 0$, and so $T(x_n)\overset{||.||}\longrightarrow 0$. Therefore, $T$ is $\mathcal{S}$-Lwc.

$2)\Rightarrow 3)$  Obvious.

$3)\Rightarrow 4)$ Obvious.

$4)\Rightarrow 1)$ Let $(x_n)$ be a disjoint weakly null sequence of $E$. Put  $K=\left\lbrace x_{n}: n\in \mathbb{N } \right\rbrace $, we note that $K$ is relatively weakly compact  and the sequence $(x_n)$ is uaw-null, hence by our
hypothesis $T(x_n)\overset{||.||}\longrightarrow 0$, as desired.

\end{proof}

As consequences of  Proposition \ref{prop0001} and  Proposition \ref{p4}, we have the following results which present new characterizations of almost Dunford-Pettis operators  which are uaw-continuous.

\begin{corollary}\label{p66}
	For a uaw-continuous operator $T:E\longrightarrow F$, the following statements are equivalent:
	\begin{enumerate}	
	\item  $T$ is almost Dunford-Pettis.
\item  For every relatively weakly compact subset $\mathcal{S}$ of $E$ and for every   uaw-compact subset $A$ of $sol(\mathcal{S})$, we have $T(A)$ is compact.
\item  For every relatively weakly compact subset $\mathcal{S}$ of $E$ and for every  relatively  uaw-compact subset $A$ of $sol(\mathcal{S})$, we have $T(A)$ is relatively compact.
	\end{enumerate}
\end{corollary}
New characterizations of the positive Schur property are obtained as a consequence of Proposition \ref{p4} and Corollary \ref{p66}.
	\begin{corollary}\label{schur}
 The following statements are equivalent:
\begin{enumerate}
	\item $E$ has the positive Schur property.	
	\item Every relatively weakly compact subset of $E$ is L-weakly compact.
	\item For every relatively weakly compact subset $\mathcal{S}$ of $E$ and for every  net $(x_{\alpha}) \subset sol(\mathcal{S})$ such that $x_{\alpha}\overset{uaw}\longrightarrow 0$, we have $x_{\alpha}\overset{||.||}\longrightarrow 0$.
	\item For every relatively weakly compact subset $\mathcal{S}$ of $E$ and for every  sequence $(x_n) \subset sol(\mathcal{S})$ such that $x_n\overset{uaw}\longrightarrow 0$, we have $x_n\overset{||.||}\longrightarrow 0$.	
	\item  For every relatively weakly compact subset $\mathcal{S}$ of $E$ and for every    uaw-compact subset $A$ of $sol(\mathcal{S})$, we have $A$ is  compact.
		\item  For every relatively weakly compact subset $\mathcal{S}$ of $E$ and for every  relatively uaw-compact subset $A$ of $sol(\mathcal{S})$, we have $A$ is relatively compact.
	\end{enumerate}
\end{corollary}

The last result of this section presents a set characterization of Banach lattices with weakly sequentially continuous lattice operations.
\begin{proposition}\label{prrrr}
   The following statements are equivalent:
	\begin{enumerate}
		\item The lattice operations of $E$ are weakly sequentially continuous.
		\item  Every relatively  weakly compact subset of $E$ is relatively uaw-compact.
	\end{enumerate}
\end{proposition}
\begin{proof}
	$1)\Rightarrow 2)$ It follows from \cite[Corollary 3.5] {k}.

	$2)\Rightarrow 1)$ Let $T:E\longrightarrow \ell^{\infty}$ be an almost Dunford-Pettis operator and let $S$ be a relatively  weakly compact subset of $E$,
by our hypothesis we have $S$ is a relatively uaw-compact set of $E$. Now  by  Proposition \ref{p4} we see that $T$ is $S$-Lwc, therefore by  Proposition \ref{p} the set $T(S)$ is relatively compact. That is, $T$ is a Dunford-Pettis operator, and so by \cite[Corollary 2.4] {BEl} the lattice operations of $E$ are weakly sequentially continuous,
and the proof of the theorem is finished.
\end{proof}


\end{document}